\documentclass[11pt]{article}

\usepackage[T1]{fontenc}
\usepackage{lmodern}
\usepackage{amsmath,amssymb,amsthm}
\usepackage{microtype}
\usepackage[margin=3.05cm]{geometry}
\usepackage{xcolor}
\usepackage{xurl}
\usepackage[
  colorlinks=true,
  linkcolor=blue!50!black,
  citecolor=blue!50!black,
  urlcolor=blue!50!black
]{hyperref}
\hypersetup{
  pdftitle={The maximum length of a chess game under the 2023 FIDE Laws},
  pdfauthor={Junyeop Yim},
  pdfkeywords={chess, FIDE Laws of Chess, seventy-five-move rule,
    fivefold repetition, longest chess game}
}
\newtheorem{theorem}{Theorem}[section]
\newtheorem{lemma}[theorem]{Lemma}
\newtheorem{proposition}[theorem]{Proposition}
\newtheorem{corollary}[theorem]{Corollary}
\theoremstyle{definition}
\newtheorem{definition}[theorem]{Definition}

\newcommand{\B}{\mathtt{B}}
\newcommand{\W}{\mathtt{W}}

\title{The maximum length of a chess game\\
under the 2023 FIDE Laws}
\author{Junyeop Yim\\
\small Department of Applied Mathematics, Kongju National University\\
\small \texttt{junyeobe0315@smail.kongju.ac.kr}}
\date{14 August 2026}

\begin{document}
\maketitle

\begin{abstract}
The FIDE Laws of Chess effective from 1 January 2023 terminate a game
automatically upon fivefold repetition or after 75 consecutive moves by each
player without a pawn move or a capture. Legal games of $17{,}697$ plies 
were previously known, and a scheduling analysis of the pawn
moves and captures gave an arithmetic upper bound of $17{,}699$ plies, but
games of $17{,}698$ or $17{,}699$ plies were not excluded.
We close this gap. Partitioning play at pawn moves
and captures yields at most $118$ segments. Each segment has length at most
$150$ plies, and each change in the colour of successive segment endpoints
reduces this bound by one ply. We prove that a game with all $118$ segments
must incur at least three such changes. Hence every game has at most
$150\cdot118-3=17{,}697$ plies, and the known constructions are optimal.
Moreover, in every maximum-length game, all sixteen pawns make six one-rank
moves and promote.
\end{abstract}

\section{Introduction}\label{sec:intro}

How long can a game of chess be under the mandatory termination rules? Before
2014, the answer depended on whether a player chose to claim a draw. Threefold
repetition and the fifty-move rule were claim-based, so cooperating players
could decline to claim and continue indefinitely; classical constructions of
unending play were given by Euwe~\cite{euwe1929} and by Morse and
Hedlund~\cite{morsehedlund1944}. On 1 July 2014, FIDE introduced automatic
draws after fivefold repetition and after 75 moves by each player without a
pawn move or a capture, making the extremal problem finite~\cite{fide2014laws}.
This paper uses the version of the Laws effective from 1 January
2023~\cite{fide2023laws}, which remains in force at the time of writing.

The bookkeeping behind the classical problem was developed by chess
problemists. Bonsdorff, Fabel, and Riihimaa~\cite{bonsdorff1974} summarise the
count of pawn moves and captures that can postpone a draw under the
fifty-move rule.
Heimo~\cite{heimo2004} observed that replacing 50 by 75 in this count gives
$17{,}697$ plies. After automatic termination was introduced,
Labelle~\cite{labelle2015longest} and Murphy~\cite{murphy2020longest} exhibited
legal games of this length. Murphy's critical-move scheduling analysis produced
an arithmetic upper bound of $17{,}699$ plies and left open whether endpoint
parity
forces a deficit of at least three plies in every game with $118$
segments~\cite[Secs.~2.1 and~3]{murphy2020longest}. We prove that it does,
thereby closing the two-ply gap.

Throughout, we disregard time forfeiture and optional endings such as
resignation, draws by agreement, and claimable draws. Play is continued from the
standard initial position until the first mandatory termination condition under
the 2023 Laws applies.

\begin{theorem}\label{thm:main}
Under the FIDE Laws of Chess effective from 1 January 2023, every legal move
sequence from the standard initial position, continued until the first
mandatory termination condition applies, has length at most $17{,}697$ plies.
This bound is attained.
\end{theorem}

Equivalently, a maximum-length game consists of $8{,}848$ complete move pairs
followed by one additional White ply.

The proof of the upper bound reduces to three estimates. Write $L$ for the
game length. A \emph{critical move} is a pawn move or a capture, and critical
moves divide a game into segments. If $K$ is the number of segments and $S$
counts changes in the colour of successive segment endpoints, with the start
of play assigned
Black parity, then
\begin{equation}\label{eq:outline}
  L\le150K-S,\qquad
  K\le118,\qquad
  K=118\Longrightarrow S\ge3.
\end{equation}
The first estimate follows from the seventy-five-move rule and the parity of
segment endpoints. The second couples the number of pawn moves to the number
of pawn captures and then counts the capturable initial pieces. Equality in the
second estimate forces all sixteen pawns to make six moves and at least $29$
captures to occur. These conditions exclude every critical-move block pattern
with at most three blocks and yield the third estimate.

Sections~\ref{sec:setting}--\ref{sec:switches} establish the three estimates in
\eqref{eq:outline}. Section~\ref{sec:maximum} establishes sharpness, completes
the proof, and records the structure forced at equality.

\section{Rules, conventions, and the segment decomposition}
\label{sec:setting}

\subsection{Rule set and chess conventions}

For the purposes of this paper, a \emph{game} is a sequence of legal moves from
the standard initial position, continued until the first mandatory termination
condition in the 2023 FIDE Laws applies. Write $L$ for its length in
\emph{plies}, where one ply is one move by one player; White plays the
odd-numbered plies. The mandatory conditions relevant to the statement are as
follows.

\begin{enumerate}
  \item Checkmate (Article~5.1.1), stalemate (Article~5.2.1), and a dead
    position (Article~5.2.2) terminate the game immediately. Checkmate and
    stalemate are positions in which the player to move has no legal move,
    the king of that player being under attack in the case of checkmate and
    not under attack in the case of stalemate. A dead position is one from
    which neither player can checkmate by any possible sequence of legal
    moves.
  \item Fivefold repetition (Article~9.6.1) terminates the game: play ends
    when the same position has appeared at least five times. For this
    purpose, positions are considered the same if and only if the same player
    has the move, pieces of the same kind and colour occupy the same squares,
    and the possible moves of all pieces of both players are the same. In
    particular, a change in the availability of a legal en passant capture or
    in castling rights makes the positions different (Article~9.2).
  \item The seventy-five-move rule (Article~9.6.2) terminates the game after
    75 consecutive moves by each player, equivalently 150 consecutive plies,
    without a pawn move or a capture. If the final ply delivers checkmate,
    checkmate takes precedence.
\end{enumerate}

We disregard time forfeiture, resignation, draws by agreement, and claimable
draws, since none can increase the number of plies. The upper-bound argument
uses Article~9.6.2 throughout and Article~5.2.2 only in
Lemma~\ref{lem:ct30}; the other mandatory termination rules can only shorten a
game. Apart from these rules, the proof uses only alternation of play and the
elementary facts about the pieces recorded below.

We use the standard coordinates of a chessboard: the files are
$a,b,\ldots,h$ and the ranks are $1,2,\ldots,8$. Each side begins with
sixteen pieces: a king, a queen, two rooks, two bishops, two knights, and
eight pawns (Articles~2.2--2.3). Throughout the proof, \emph{piece} includes
a pawn. A \emph{capture} is a move to a square occupied by an opposing
piece; the captured piece is removed from the board as part of the same
move and takes no further part in the game (Article~3.1.1). The en passant
capture, described below, is the one exception, in which the captured piece
does not stand on the destination square.

White pawns begin on rank~2 and Black pawns on rank~7, one pawn of each
colour on each file; White pawns move towards rank~8 and Black pawns
towards rank~1. We call ranks~2 and~7 their respective home
ranks, and ranks~8 and~1 their promotion ranks. A pawn advances one rank to
an unoccupied square or, on its first move from its home rank, two ranks
provided that both squares ahead of it are unoccupied
(Articles~3.7.1--3.7.2). It captures one square diagonally forward onto a
square occupied by an opposing piece (Article~3.7.3). In the en passant
capture (Articles~3.7.3.1--3.7.3.2), immediately after an opposing pawn has
advanced two ranks from its home rank, a pawn may capture it as though it
had advanced only one rank, moving diagonally forward onto the empty square
that the opposing pawn passed over. While it remains a pawn, it changes
file only by making a capture. Upon reaching the promotion rank a pawn
must, as part of the same move, be exchanged for a new queen, rook, bishop,
or knight of its own colour (Article~3.7.3.3); after promotion it is no
longer a pawn.

A piece \emph{attacks} a square if the movement rules would allow it to
capture an opposing piece standing there, even when such a move is
otherwise illegal (Articles~3.1.2--3.1.3); a king is \emph{under attack}
when an opposing piece attacks the square it occupies (Article~3.9.1).
The king moves one square in any direction (Article~3.8.1), so it attacks
precisely the adjacent squares.

Kings are never captured, since capturing a king is not allowed
(Article~1.4.1); hence each side has fifteen capturable initial pieces:
eight pawns and seven non-pawn pieces. We follow each of the $32$ initial
pieces through the game, retaining its identity after promotion. Thus a
promoted pawn remains one of the eight pieces that began as pawns, and
promotion creates no additional capturable identity.

\subsection{Critical moves and segments}

\begin{definition}[Critical and noncritical moves]\label{def:critical}
A \emph{critical move} is a pawn move or a capture. A
\emph{noncritical move} is any other move. Thus an en passant capture is
critical for both reasons, promotion is critical because it is a pawn move, and
castling\footnote{Castling (Article~3.8.2) is a combined move of the king and
one rook of the same colour; it counts as a single move and captures
nothing.} is noncritical.
\end{definition}

Critical moves are precisely the moves that reset the counter in
Article~9.6.2.

\begin{definition}[Segments and endpoints]\label{def:segments}
Let
\[
  t_1<t_2<\cdots<t_m
\]
be the plies on which the critical moves of a game occur. Set $T=1$ if at
least one noncritical ply follows the last critical move, or if the game
contains no critical move, and set $T=0$ otherwise. Define
\[
  e_0=0,\qquad e_i=t_i\quad(1\le i\le m),
\]
and, when $T=1$, define $e_{m+1}=L$. Put $K=m+T$. For
$1\le i\le K$, the interval of plies
\[
  (e_{i-1},e_i]\cap\mathbb Z
  =\{e_{i-1}+1,\ldots,e_i\}
\]
is \emph{segment $i$}. When $T=1$, the final segment is called the
\emph{terminal noncritical segment}.
\end{definition}

For $i\ge1$, let $\operatorname{col}(e_i)$ be the colour of the player who
moves on ply $e_i$: White when $e_i$ is odd and Black when it is even. We assign the virtual endpoint $e_0=0$ the colour Black. This convention
treats the start of the game as having the parity of a preceding Black
ply, so that the parity relation between segment length and endpoint
colours in Lemma~\ref{lem:decomposition} applies to the first segment
as well: a first segment ending on a White ply has odd length.

\begin{definition}[Switch count]\label{def:switch}
The \emph{switch count} is
\[
  S=\#\bigl\{i\in\{1,\ldots,K\}:
    \operatorname{col}(e_i)\ne\operatorname{col}(e_{i-1})\bigr\}.
\]
\end{definition}

For example, the following endpoint parities give three segments and one
switch:
\[
  0_{\B}\xrightarrow{150}150_{\B}
  \xrightarrow{149}299_{\W}
  \xrightarrow{150}449_{\W}.
\]
The total length is therefore $150\cdot3-1=449$.

\begin{lemma}[Segment decomposition]\label{lem:decomposition}
Every game satisfies
\begin{equation}\label{eq:decomposition}
  L=150K-S-\sum_{i=1}^{K}\delta_i,
  \qquad \delta_i\in\mathbb Z_{\ge0}.
\end{equation}
Here $\delta_i$ is the shortfall of segment $i$ from length $150$ when its
endpoints have the same colour, and from length $149$ when they have different
colours. In particular,
\[
  L\le150K-S.
\]
\end{lemma}

\begin{proof}
Fix a segment with endpoints $e_{i-1}$ and $e_i$. Every ply in the segment
except possibly its last is noncritical; in a terminal noncritical segment, the
last ply is noncritical as well. If the segment had length at least $151$,
then its first $150$ plies would be noncritical, and Article~9.6.2 would terminate the game at ply
$e_{i-1}+150$. Hence every segment has length at most $150$.

Because plies alternate colours, $e_i-e_{i-1}$ is even exactly when the two
endpoints have the same colour. A segment whose endpoints have different
colours therefore has odd length and hence length at most $149$. Subtracting
each segment length from the corresponding bound defines
$\delta_i\in\mathbb Z_{\ge0}$. Summing over the segments gives
\eqref{eq:decomposition}.
\end{proof}

Let $P$ be the number of pawn moves, $C$ the number of captures, and
$C_{\mathrm p}$ the number of captures made by pawns. The number of critical
moves is $P+C-C_{\mathrm p}$, and therefore
\begin{equation}\label{eq:K}
  K=(P-C_{\mathrm p})+(C+T).
\end{equation}

\section{The 118-segment bound}\label{sec:k118}

\begin{lemma}[Pawn moves]\label{lem:pawnbasics}
Each pawn makes at most six pawn moves. Consequently $P\le96$. If a pawn
makes six pawn moves, every one of them advances it by exactly one rank, and
the sixth move promotes it.
\end{lemma}

\begin{proof}
The promotion rank is six rank-steps from the home rank, and every pawn move
advances the pawn by at least one rank towards promotion. Equality therefore
requires six one-rank moves, the last of which promotes the pawn.
\end{proof}

\subsection{Pawn moves and pawn captures}

\begin{definition}[Origin pairs]\label{def:originpair}
For each file, the White pawn and Black pawn that begin on that file form its
\emph{origin pair}. An origin pair is \emph{resolved} if at least one of its
members makes a capture while still a pawn, and \emph{unresolved} otherwise.
A capture made on the promotion move counts as a capture by a pawn; a
capture made by the promoted piece on a later move does not. Let $f$ be
the number of resolved origin pairs.
\end{definition}

\begin{lemma}[Unresolved origin pair]\label{lem:cap10}
The two pawns in an unresolved origin pair make at most ten pawn moves in
total.
\end{lemma}

\begin{proof}
While both pawns remain on the board as pawns, neither changes file, so they
remain on their common origin file with the White pawn strictly below the
Black pawn: straight pawn advances require unoccupied squares, so the two
pawns can neither meet nor pass one another. If they have advanced $a$ and
$b$ ranks, respectively, then
\[
  2+a<7-b,
\]
and hence $a+b\le4$. Thus the pair makes at most four pawn moves while both
members remain on the board. In particular, neither pawn can promote before
the other is captured.

If neither pawn is ever captured, their combined total is at most four.
Otherwise, the pawn captured first has made at most four moves, while the
surviving pawn makes at most six moves over its entire lifetime by
Lemma~\ref{lem:pawnbasics}. The combined total is therefore at most
$4+6=10$.
\end{proof}

\begin{lemma}[Pawn-capture lower bound]\label{lem:pawncapturefloor}
$C_{\mathrm p}\ge f$.
\end{lemma}

\begin{proof}
For each resolved origin pair, choose the first capture made by one of its
pawns. This move contributes to $C_{\mathrm p}$. Different origin pairs give
different moving pawns and hence different moves.
\end{proof}

\begin{proposition}[Net pawn-move bound]\label{prop:pminusc}
For every game,
\[
  P-C_{\mathrm p}\le88.
\]
Equality implies
\[
  f=8,\qquad P=96,\qquad C_{\mathrm p}=8.
\]
\end{proposition}

\begin{proof}
For each unresolved origin pair, Lemma~\ref{lem:cap10} gives at most ten pawn
moves. For each resolved origin pair, Lemma~\ref{lem:pawnbasics} gives at
most twelve. Hence
\[
  P\le\min\bigl\{96,10(8-f)+12f\bigr\}
   =\min\{96,80+2f\}.
\]
Together with $C_{\mathrm p}\ge f$, this gives
\[
  P-C_{\mathrm p}\le\min\{96,80+2f\}-f.
\]
For $f\le7$, the right-hand side is at most $80+f\le87$; for $f=8$, it is at
most $96-8=88$. Equality therefore requires $f=8$, after which the bounds
$P\le96$ and $C_{\mathrm p}\ge8$ force $P=96$ and
$C_{\mathrm p}=8$.
\end{proof}

\subsection{Captures and the terminal noncritical segment}

\begin{lemma}[Capture--terminal bound]\label{lem:ct30}
$C+T\le30$.
\end{lemma}

\begin{proof}
There are $30$ capturable initial pieces, so $C\le30$. If $C\le29$, then
$C+T\le30$ because $T\le1$. If $C=30$, only the two kings remain after the
last capture. A position containing only the two kings is
dead,\footnote{With only the kings on the board, neither king can be
under attack: the only possible attacker is the opposing king, and the two kings
can never stand on adjacent squares, because the move creating such a
position would place the moving player's own king under attack, which is
illegal (Article~3.9.2). Since checkmate requires the king to be under
attack, no sequence of legal moves leads to checkmate.} so by
Article~5.2.2 no ply can follow that capture. Hence no terminal
noncritical segment exists, $T=0$, and again $C+T=30$.
\end{proof}

\subsection{The segment bound and its equality case}

\begin{theorem}[Segment bound]\label{thm:k118}
Every game satisfies $K\le118$. If $K=118$, then
\begin{equation}\label{eq:equalityprofile}
  f=8,\qquad P=96,\qquad C_{\mathrm p}=8,\qquad C+T=30.
\end{equation}
In particular, all sixteen pawns make six one-rank moves, promoting on the
sixth, and $C\ge29$.
\end{theorem}

\begin{proof}
By \eqref{eq:K}, Proposition~\ref{prop:pminusc}, and Lemma~\ref{lem:ct30},
\[
  K=(P-C_{\mathrm p})+(C+T)\le88+30=118.
\]
Equality requires equality in both component bounds.
Proposition~\ref{prop:pminusc} gives $f=8$, $P=96$, and
$C_{\mathrm p}=8$, while Lemma~\ref{lem:ct30} gives $C+T=30$. Since
$T\le1$, we have $C\ge29$. Finally, $P=96$ forces every pawn to attain the
equality case of Lemma~\ref{lem:pawnbasics}.
\end{proof}

\section{The three-switch bound}\label{sec:switches}

Throughout this section, assume $K=118$. By Theorem~\ref{thm:k118}, every
pawn makes six pawn moves and at least $29$ captures occur. Since $T\le1$, the
game has at least $117$ critical moves; in particular, the critical-move
colour sequence is nonempty.

\subsection{Block patterns}

\begin{definition}[Block pattern]\label{def:pattern}
For a nonempty finite sequence with entries in $\{\B,\W\}$, its \emph{block
pattern} is obtained by replacing each maximal constant run by a single
letter; these maximal runs are the \emph{blocks} of the sequence. The
\emph{critical-move colour sequence} is
\[
  \bigl(\operatorname{col}(t_1),\ldots,
        \operatorname{col}(t_m)\bigr),
\]
and the \emph{endpoint colour sequence} is
\[
  \bigl(\operatorname{col}(e_1),\ldots,
        \operatorname{col}(e_K)\bigr).
\]
Their block patterns are called the \emph{critical-move block pattern} and the
\emph{endpoint block pattern}, respectively.
\end{definition}

For example,
\[
  (\B,\B,\W,\W,\B)\longmapsto\B\W\B.
\]
The virtual endpoint $e_0$ is not part of the endpoint colour sequence; it is
used only in the definition of $S$.

\begin{lemma}[Endpoint blocks]\label{lem:blockrelation}
The endpoint colour sequence has at least as many blocks as the critical-move
colour sequence. If the endpoint colour sequence has $r$ blocks, then
\[
  S=
  \begin{cases}
    r-1,&\text{if the first endpoint is Black},\\
    r,&\text{if the first endpoint is White}.
  \end{cases}
\]
Consequently, an endpoint colour sequence with at least four blocks has
$S\ge3$.
\end{lemma}

\begin{proof}
The endpoint colour sequence is the critical-move colour sequence, with one
terminal endpoint appended when $T=1$. Appending an element cannot reduce the
number of blocks. The formula for $S$ counts the $r-1$ changes between
consecutive blocks and, when the first endpoint is White, the additional
change from the virtual Black endpoint $e_0$.
\end{proof}

The possible critical-move block patterns with at most three blocks are
\[
  \B,\quad \W,\quad \B\W,\quad \W\B,\quad
  \B\W\B,\quad \W\B\W.
\]
We exclude them in turn. For $\W\B\W$ this is more than the inequality
$S\ge3$ requires: its initial White endpoint already incurs a switch against
the virtual Black endpoint $e_0$, so Lemma~\ref{lem:blockrelation} gives
$S\ge3$ directly. The exclusion nevertheless matters, because the four-block
conclusion of Theorem~\ref{thm:s3} is what determines the block pattern of
maximum-length games in Corollary~\ref{cor:equality}.

\subsection{One and two blocks}

\begin{proposition}\label{prop:oneblock}
The block patterns $\B$ and $\W$ are impossible when $K=118$.
\end{proposition}

\begin{proof}
A pawn move is a critical move of the pawn's side. With only one block, only
one side's eight pawns can move, so $P\le8\cdot6=48$, contrary to $P=96$.
\end{proof}

\begin{proposition}\label{prop:twoblocks}
The block patterns $\B\W$ and $\W\B$ are impossible when $K=118$.
\end{proposition}

\begin{proof}
Write the pattern as $(X,Y)$, so every critical move by $X$ precedes every
critical move by $Y$. Let $C_X$ and $C_Y$ be the respective numbers of
captures. Since $Y$ can capture at most the fifteen capturable pieces of $X$,
we have $C_Y\le15$, and therefore
\[
  C_X=C-C_Y\ge29-15=14.
\]
All these captures occur before the first critical move by $Y$. Among the
captured $Y$-pieces, at most seven did not begin as pawns, so at least seven
$Y$-pawns are captured before making any pawn move. They cannot move after
they are captured. Hence
\[
  P\le96-7\cdot6=54,
\]
contrary to $P=96$.
\end{proof}

\subsection{A home-rank obstruction}

\begin{lemma}[Home-rank lemma]\label{lem:homerank}
Let $X$ and $Y$ be the two opposing sides. Consider a prefix of a legal game
from the standard initial position. Suppose that
\begin{enumerate}
  \item every move made by $Y$ in this prefix is noncritical;
  \item no move made by $X$ in this prefix captures a $Y$-pawn.
\end{enumerate}
Then all eight $Y$-pawns remain on their home squares, no $X$-pawn reaches
$Y$'s pawn home rank, and every $X$-pawn makes at most four pawn moves during
the prefix.
\end{lemma}

\begin{proof}
A move by a $Y$-pawn would be critical, so the first condition implies that
no $Y$-pawn has moved. The second condition implies that none has been
captured. Thus every $Y$-pawn remains on its initial square.

An $X$-pawn cannot move onto $Y$'s pawn home rank. A forward move to an
occupied square is illegal, while a diagonal move to such a square would
capture the $Y$-pawn, contrary to the second condition. Between an $X$-pawn's
home rank and $Y$'s pawn home rank there are four ranks on which it may stand.
Every pawn move advances it by at least one rank, so it makes at most four
pawn moves during the prefix.
\end{proof}

\subsection{Three blocks}

The remaining case is ruled out by the home-rank obstruction.

\begin{proposition}\label{prop:threeblocks}
The block patterns $\B\W\B$ and $\W\B\W$ are impossible when $K=118$.
\end{proposition}

\begin{proof}
Write the pattern as $(X,Y,X)$, and let $C_X,C_Y$ be the numbers of captures
made by the two sides. Since $X$ can capture at most fifteen $Y$-pieces,
\[
  C_Y=C-C_X\ge29-15=14.
\]
At most seven of the pieces captured by $Y$ did not begin as pawns. Hence at
least seven of the captured $X$-pieces began as pawns.

Let $\tau$ be the ply on which $Y$ makes its first critical move, and fix
one such piece $p$,
which began as an $X$-pawn. By Theorem~\ref{thm:k118}, $p$ makes six pawn
moves before it is captured. Any critical move by $X$ after $\tau$ belongs to
the final $X$-block and therefore occurs after every critical move in the
middle $Y$-block, including the capture of $p$. Consequently all six pawn
moves of $p$ occur before $\tau$.

Consider the game prefix ending immediately before $\tau$. Every move made by
$Y$ in this prefix is noncritical. Moreover, $X$ cannot have captured a
$Y$-pawn:
such a pawn would have been removed before making any critical move and could
not make the six pawn moves required by $P=96$. Lemma~\ref{lem:homerank}
therefore applies to this prefix and implies that $p$ makes at most four pawn
moves before $\tau$, contradicting the six moves just established.
\end{proof}

\begin{theorem}[Three-switch bound]\label{thm:s3}
If $K=118$, then the critical-move colour sequence has at least four blocks
and $S\ge3$.
\end{theorem}

\begin{proof}
Propositions~\ref{prop:oneblock},~\ref{prop:twoblocks}, and~\ref{prop:threeblocks}
exclude all nonempty block patterns with at most three
blocks. Hence the critical-move colour sequence has at least four blocks.
Lemma~\ref{lem:blockrelation} transfers this lower bound to the endpoint
colour sequence and gives $S\ge3$.
\end{proof}

\section{Completion of the proof and the equality case}
\label{sec:maximum}

\begin{proposition}[Sharpness]\label{prop:sharpness}
There exists a legal move sequence of $17{,}697$ plies from the standard
initial position such that no mandatory termination condition applies before
the final ply, and the final ply delivers checkmate.
\end{proposition}

\begin{proof}
Labelle~\cite{labelle2015longest} and Murphy~\cite{murphy2020longest} gave
explicit legal games of this length. Murphy's move sequence has been verified
by replay~\cite{repo}: the replay checks move legality, fivefold repetition,
and the seventy-five-move rule after every ply, and confirms checkmate on the
final ply. It includes no decision procedure for dead positions under
Article~5.2.2, but no such procedure is needed for this particular game.
From every
earlier position, the remaining verified suffix is a legal sequence ending in
checkmate, so that position is not dead; the existence of the next legal ply
also rules out earlier checkmate and stalemate. Hence no mandatory
termination condition applies before the final ply, and the lower bound
follows.
\end{proof}

\begin{proof}[Proof of Theorem~\ref{thm:main}]
Let $K$ be the number of segments. By Theorem~\ref{thm:k118}, $K\le118$. If
$K\le117$, then Lemma~\ref{lem:decomposition} gives
\[
  L\le150\cdot117=17{,}550.
\]
If $K=118$, then Theorem~\ref{thm:s3} gives $S\ge3$, and therefore
\[
  L\le150\cdot118-3=17{,}697.
\]
This proves the universal upper bound. Proposition~\ref{prop:sharpness} shows
that it is attained.
\end{proof}

\begin{corollary}[Structure of maximum-length games]\label{cor:equality}
Every game of length $17{,}697$ has
\[
  K=118,\qquad S=3,\qquad \delta_i=0\quad(1\le i\le118),
\]
and
\[
  f=8,\qquad P=96,\qquad C_{\mathrm p}=8,\qquad C+T=30.
\]
Thus all sixteen pawns make six one-rank moves, promoting on the sixth; at
least $29$ captures occur; and every segment has the maximum length permitted
by its endpoint parity.
The critical-move and endpoint block patterns are both $\B\W\B\W$.
\end{corollary}

\begin{proof}
A game with $K\le117$ has length at most $17{,}550$, so a maximum-length game
has $K=118$. Combining Theorem~\ref{thm:s3} with
\eqref{eq:decomposition} gives
\[
  17{,}697
  =150\cdot118-S-\sum_{i=1}^{118}\delta_i
  \le150\cdot118-3.
\]
Equality holds throughout. Hence $S=3$ and every $\delta_i=0$. The equality
statement in Theorem~\ref{thm:k118} gives the remaining numerical assertions
and the behaviour of the pawns.

The critical-move colour sequence has at least four blocks by
Theorem~\ref{thm:s3}. Each transition between consecutive blocks of this
sequence is a change of colour between consecutive endpoints and contributes
one switch, so $S=3$ permits at most four blocks.
Thus there are exactly four. If the first block were White, the comparison
with the virtual Black endpoint would contribute a fourth switch. Hence the
critical-move block pattern is $\B\W\B\W$. Appending a terminal endpoint
cannot create a new block without increasing $S$, so the endpoint block
pattern is the same.
\end{proof}

\section*{Data and code availability}

The companion archive~\cite{repo} contains Murphy's $17{,}697$-ply move
sequence, the two separately implemented replay verifiers that perform the
checks described in the proof of Proposition~\ref{prop:sharpness}, and
ancillary computational cross-checks. None of these computations enters the
upper-bound proof. Verification instructions accompany the archive; the
development version and detailed verification notes are maintained in the
\href{https://github.com/junyeobe0315/longest-chess-game}{project repository on
GitHub}.

\section*{Acknowledgements}

Tom Murphy VII's construction and question motivated this work, and he
kindly commented on a draft, as did Alexis Langlois-R\'emillard. The author is
deeply grateful to Jinwan Park for generously taking the time to discuss this
work and for valuable advice and encouragement. The author also thanks
Fran\c{c}ois Labelle for his detailed historical account of the
longest-game problem.

\end{document}